\documentclass[11pt]{article}
\usepackage{amsmath, amssymb, amsthm, verbatim,enumerate,bbm}
\usepackage{indentfirst}

\title{Packing perfect matchings in random hypergraphs}

\author{ Asaf Ferber
\thanks{Department of Mathematics, Yale University, and Department of Mathematics, MIT. Emails:
asaf.ferber@yale.edu, and ferbera@mit.edu.}\and Van Vu
\thanks{Department of Mathematics, Yale University. Email: van.vu@yale.edu. Supported by   NSF  grant DMS-1307797  and AFORS grant
FA9550-12-1-0083.}}

\date{\today}
\allowdisplaybreaks

\theoremstyle{plain}
\newtheorem{theorem}{Theorem}[section]

\newtheorem{claim}[theorem]{Claim}

\newtheorem{observation}[theorem]{Observation}

\newtheorem{remark}[theorem]{Remark}

\newtheorem{question}[theorem]{Question}

\begin{document}
\maketitle

\begin{abstract}
  We introduce a  new procedure for generating the binomial random
  graph/hypergraph models, referred to as \emph{online
  sprinkling}. As an illustrative application of this method, we show that for any
  fixed integer $k\geq 3$, the binomial $k$-uniform random hypergraph
  $H^{k}_{n,p}$ contains $N:=(1-o(1))\binom{n-1}{k-1}p$ edge-disjoint
  perfect matchings, provided $p\geq \frac{\log^{C}n}{n^{k-1}}$,
  where $C:=C(k)$ is an integer depending only on $k$. Our result for $N$ is asymptotically optimal and for
   $p$ is optimal up to the $polylog(n)$ factor. This significantly improves a result of Frieze and Krivelevich.
\end{abstract}

\section{Introduction}
Since its introduction in 1960~\cite{ErdosRenyi}, the Erd{\H
o}s-R{\'e}nyi random graph/hypergraph model has been one of the main
objects of study in probabilistic combinatorics. Given $p\in [0,1]$
and $k\in \mathbb{N}$, the random $k$-uniform hypergraph model
$H^k_{n,p}$ is defined on a vertex set $[n]:=\{1,\ldots,n\}$,
obtained by picking each $k$-tuple $e\in \binom{[n]}{k}$ to be an
edge independently with probability $p$. The case $k=2$ reduces to
the standard binomial graph model, denoted as $G_{n,p}$.

A useful technique in the theory of random graphs is the \emph{multiple exposure} technique
(also referred to as \emph{sprinkling}). Given $p_1,\ldots,p_{\ell}\in
[0,1]$ for which $\prod_{i=1}^{\ell}(1-p_i)=1-p$, one can easily
show that a hypergraph $H^k_{n,p}$ has the same distribution as
a union of independently generated hypergraphs $H=H_1\cup \ldots
\cup H_{\ell}$, where for each $i$, $H_i=H^k_{n,p_i}$
(for more details, the reader is referred to \cite{bollobas1998random}, \cite{JLR} or \cite{S} for a more relevant approach). Indeed, note that the probability for a fixed $k$-tuple $e\in \binom{[n]}{k}$ to \emph{not appear} in $\cup_iE(H_i)$ is exactly $\prod_{i=1}^{\ell}(1-p_i)=1-p$, and clearly, all the choices are being made independently.

The power of this technique comes from the ability to ``keep some randomness" in cases where an iterative approach is convenient. A typical scenario in applications is to expose $H^k_{n,p}$ in stages, where in each stage, a hypergraph $H_i=H^k_{n, p_i}$ is being generating, independently at random from all the previously exposed hypergraphs. Our goal is to show that in each stage $j$, the current hypergraph $\cup_{i\leq j}H_i$ gets
closer to a target graph property $\mathcal P$, until in stage $\ell$ it satisfies it.
This technique has became standard over the years and is being used in almost every paper dealing with random graphs/hypergraphs (for a very nice and classical example, the reader is referred to \cite{bollobas1984evolution} and \cite{posa1976hamiltonian}).

In this paper we want to consider a slightly different perspective of the sprinkling method, which gives it a bit more power. Before describing it, let us have a closer look at the way a hypergraph $H^k_{n,p}$ is being generated. By definition, for every $k$-tuple $e\in \binom{[n]}{k}$, we query whether $e\in E(H)$ with probability $p$, independently at random. A natural question arises is:

\begin{question}Does the order of the edge queries matter ?
\end{question}

Clearly, the answer is ``no", as long as all the queries are being made independently at random, and this observation serves as the basis for our technique.

Our goal is to create a  randomized algorithm  that whp (with high probability, that is, with probability tending to 1 as $n$ tends to infinity) finds a large structure $S$ in $H^k_{n,p}$. We aim to find the target structure as a subgraph of the ``online generated" random hypergraph $H$. That is, during the execution of the algorithm, a random hypergraph is
generated and the target structure is constructed together \emph{step by step}. We refer to this technique as \emph{online sprinkling}.

The way the algorithm works is as follows: in each time step $i$ of the algorithm, a subset $E_i\subseteq \binom{[n]}{k}$ is being chosen
according to some distribution. Then, we query every edge in $E_i$ (independently) with some probability $p_i$, which is also being chosen according to some distribution. All the chosen edges will be part of the randomly generated hypergraph.

For each $k$-tuple $e\in \binom{[n]}{k}$, let
$$\omega(e)=1-\prod_{i: e\in E_i}(1-p_i)$$ be the \emph{weight} of
$e$ at the end of the algorithm. Note that $\omega(e)$ is a random variable (as our algorithm is a randomized one), and corresponds to the probability for $e$ to appear in the hypergraph obtained at the end of the algorithm. Clearly, if $\omega(e)\leq p$ for
each $k$-tuple $e$, then the resulting hypergraph can be coupled as
a subgraph of $H^k_{n,p}$.

The power of this approach comes from the flexibility in defining the sets $E_i$ and the edge-probabilities $p_i$. By selecting these sets and probabilities properly, we can govern the process
towards our goal.

As an illustrative example for this approach, we examine the problem of finding edge-disjoint copies of some given structure $S$ in $H^k_{n,p}$. In particular, in this paper we consider only the case where $S$ is a \emph{perfect matching} and $k\geq 3$, as here, many technical issues that may appear for other structures $S$ or for the case $k=2$ will become trivial (more complicated applications will appear in followup papers).

The problem of finding the threshold behavior for the appearance of  a perfect matching in a random
hypergrpah is notoriously hard and is a central problem in probabilistic combinatorics, known as
Shamir's Conjecture. The main difficulty is the lack of general tools such as the classical theorem by Hall (see e.g.,\cite{West}) for finding perfect matchings. This problem was solved by Johansson, Kahn and
Vu \cite{JKV}, who showed that a perfect matching typically appears in $H^k_{n,p}$ as soon as $p\geq
\frac{C\log n}{n^{k-1}}$ (note that $n$ must be divisible by $k$, as otherwise a perfect matching cannot exist). Once Shamir's Conjecture has been settled, it is thus natural to ask for
edge-disjoint perfect matchings covering ``most" of the edges. This problem has been considered by Frieze and Krivelevich in \cite{FK}, where they showed, among other things, that one can pack ``most" of the edges of a typical $H^k_{n,p}$ with perfect matchings, as long as $p>\log^2n/n$. Moreover, they showed that there embedding can be applied on a pseudorandom model with the same density. Considering only the random model, in the following theorem we significantly improve their result to the optimal (up to a $polylog(n)$ factor) edge-probability.

\begin{theorem}
  \label{packing not Kpartite}
Let $k\geq 3$ be a positive integer and let
$p\geq \frac{\log^{5k}n}{n^{k-1}}$. Then, whp
$H^k_{kn,p}$ contains
$t:=(1-o(1))\binom{kn-1}{k-1}p$ edge-disjoint perfect matchings.
\end{theorem}

\begin{remark} We would like to give the following remarks:
\begin{itemize}
\item The case $k=2$ is a bit more complicated to handle using our technique, and in fact better tools are known for this case (generalizations of Hall's Theorem for finding ``many" edge-disjoint perfect matchings). For a non-trivial example of applying the ``online sprinkling" technique for graphs, the reader is referred to \cite{FS}.
\item Our $p$ is optimal up to a $polylog(n)$ factor and $t$ is asymptotically optimal. In fact, as we explain bellow, our proof strategy will always yield a lost of some $\log$'s in $p$, and therefore we do not put any effort in optimizing its power. Even though, we believe that the same conclusion should hold for every edge-probability $p$ which is asymptotically larger than the threshold behavior.
\item Our proof heavily relies on the ability to embed \emph{one} perfect matching (that is, on the result from \cite{JKV}), and does not provide an alternative proof for that.
\end{itemize}
\end{remark}

For some technical reasons, it will be more convenient for us to work in a $k$-partite model. Let $\mathcal H^k_{n\times k,p}$ be a random $k$-partite, $k$-uniform hypergraph, with parts
$V_1,\ldots,V_k$, each of which of size $n$, obtained by adding each
possible $k$-tuple $e\in V_1\times V_2\times \ldots V_k$ as an edge with probability $p$, independently at random. We prove the following, seemingly weaker, statement about finding edge-disjoint perfect matchings in $H^k_{n\times k,p}$, and then show how to to derive Theorem \ref{packing not Kpartite} in a quite straightforward way.

\begin{theorem}
  \label{main1}
Let $k\geq 3$ be a positive integer and let
$p\geq \frac{\log^{4k}n}{n^{k-1}}$. Then, whp a
hypergraph $H^k_{n\times k,p}$ contains
$(1-o(1))n^{k-1}p$ edge-disjoint perfect matchings.
\end{theorem}

{\bf Organization of the paper.} In Section \ref{sec:outline} we provide with a brief outline of the general strategy for proving Theorem \ref{main1}, explaining the difficulties one may run into while using the ``online sprinkling" technique. In Section \ref{sec:derivation} we show how to derive Theorem \ref{packing not Kpartite} from Theorem \ref{main1}. In Section \ref{sec:tools} we present some tools and auxiliary lemmas to be used in the proof of Theorem \ref{main1}. Lastly, in Section \ref{sec:main} we prove our main result, namely Theorem \ref{main1}.

\section{A general outline}
\label{sec:outline}
Our proof, in large, is divided into two main phases. In Phase 1 we wish to find the ``correct number" of edge-disjoint matchings which are not complete, where in Phase 2 we wish to ``complete" each of which into a perfect matching in an edge-disjoint way (and this will be done using the Johansson, Kahn and Vu's result \cite{JKV}). So far, our proof strategy is not new, and in fact the exact same strategy has been used in many papers during the years (perhaps the most impressive recent result obtained by a similar outline is the one of Keevash \cite{Keevash}, where he solved a problem from the 19th century). The main idea behind it is that, usually, it is much easier to find ``almost spanning" structures than ``spanning" ones, and if one can embed the almost spanning substructure ``nicely" then there is a hope to complete it to desired spanning structure. Bellow, we give a brief description of each of the two phases, and explain the difficulties we should overcome during the formal proof.

{\bf Phase 1.} The way we handle the ``almost spanning structure" is more or less identical to the ``nibbling" idea,  introduced by
Ajtai-Komlos-Szemer\'edi \cite{AKS} and R\"odl  \cite{Rodl}. Roughly speaking, we split Phase 1 into $N$ \emph{Rounds}, where each round is being further divided into \emph{Steps}. In each Round $i$, our goal is to find a ``large" matching $M_i$. In order to do so, we start with an empty matching $M_{i0}$, and in each Step $j$, we extend the current matching $M_{i(j-1)}$ by a ``bit", while exposing edges which are vertex disjoint to $M_{i(j-1)}$. Note that the rounds run independently, while completely ignoring the history. We later show (Lemma \ref{matchings disjoint}) that if $p$ is not too large, then this procedure gives us edge-disjoint matchings whp (we then show how to deal the case where $p$ is large).

Let us focus in one round. The main observation here is that if we expose edges with ``relatively small" probability, then one can easily show (Lemma \ref{small bite exists}) that ``most" of them form a matching (edges which are overlapped with other edges will be just ignored). It is worth mentioning that the nibbling approach is typically being applied in a deterministic setting where a ``nicely behaved" (hyper)graph is given. Then, by sampling ``not too many" edges, one can easily show that most of them form a matching. Therefore, most of the work is focused in showing that the remaining set of edges is still ``nicely behaved". In our setting, as we expose the hypergraph in an online fashion, we will obtain it for free.

A crucial point during this phase, is that, as we show, due to symmetry, each $M_i$ is actually a matching chosen \emph{uniformly} at random. Letting $U_i:=V(H)\setminus \cup M_i$, we obtain $N$ sets $(U_i)_{i=1}^N$, each of which is a random subset chosen according to a uniform distribution. This fact will be useful in Phase 2.

The main problem in Phase 1 is to show that no edge has accumulated ``too much" weight. Namely, let $p_1,\ldots,p_s$ to denote all the edge-probabilities used during the algorithm in order to ``expose" a particular $k$-tuple $e$, we wish to show that $1-\prod_{i=1}^s(1-p_i)\leq (1-\varepsilon/2)p$. Assuming this, we obtain a natural coupling between the hypergraph which has been generated in this phase and a subhypergraph of $H_{n,(1-\varepsilon/2)p}$.

{\bf Phase 2.} In this phase our goal is to complete the matchings into perfect matchings in an edge-disjoint way. To this end, for each $1\leq i\leq N$, we expose all the $k$-tuples in $U_i$ with probability $q=\frac{\log^2n}{|U_i|^{k-1}}$. Then, the main result of \cite{JKV} ensures us a perfect matching in $U_i$ whp (for all $i$ simultaneously). Note that as $|U_i|$ is going to be relatively small (some natural restriction apply during the proof), it follows that one cannot hope to get the ``correct" edge-probability from our proof and there will always be a lose of few logs. Now, adding such a matching to $M_i$ yields a perfect matching of $H$. It thus remain to show that the matchings are disjoint (Lemma \ref{matchings disjoint}) and that none of the ``new added" edges accumulated more than a weight of (say) $\varepsilon p/3$. Assuming that, there is a natural coupling between the ``new" hypergraph and a subhypergraph of $H^k_{n,\varepsilon p/3}$, and therefore the union of the two hypergraphs generated in both phases, has the same distribution as a subgraph of $H^k_{n,p}$. This will complete the proof.

\section{Derivation of Theorem \ref{packing not Kpartite}}
\label{sec:derivation}
\begin{proof}
  Let $t=\log^{1.5}n$, and take $t$ partitions
  $[kn]:=V_1^{(i)}\cup \ldots \cup V_k^{(i)}$ with parts of size
  precisely $n$, independently, uniformly at random.  For each $k$
  tuple $e\in \binom{[kn]}{k}$, let us define the set of \emph{relevant} partitions for $e$ as $\mathcal R_e:=\{i\leq
  t: e\cap V_j^{(i)}\neq \emptyset \text{ for all }j\leq k\}$. Note that
  $$\Pr\left[i\in \mathcal{R}_e\right]=k!/k^{k}=\Theta(1),$$ and
  that for $i\neq j$, the events ``$i\in \mathcal R_e$" and ``$j\in \mathcal R_e$" are independent.
  Therefore, by Chernoff's bounds  one obtains  that with probability $1-e^{-\Theta(t)}$
  $$\mathcal
  R_e=(1+o(1))(k!/k^{k})t=:r$$ holds for every $k\in \binom{[kn]}{k}$.

  Now,
  expose all the $k$-tuples with probability $p$, independently at
  random, and for each tuple $e\in E(H)$, let $f(e)\in \mathcal R_e$ be a
  uniformly random element. For each $i\leq t$, let $H_i$ be the $k$-partite hypergraph with parts $V_1^{(i)}\times \ldots \times V^{(i)}_k$ obtained by
  taking all the edges $E_i:=\{e\in E(H): f(e)=i\}$, and note that
  $H_i=H^{k}_{n\times k,(1-o(1))p/r}$ (although for $i\neq j$ $H_i$ and $H_j$ are not independent!) and that for $i\neq j$, $E(H_i)\cap E(H_j)=\emptyset$.

  Fixing an $i$, by Theorem \ref{main1} it follows that whp $H_i$ contains $m=(1-o(1))n^{k-1}p/r$ edge-disjoint perfect matchings. Therefore, by applying Markov's inequality, we obtain that for all but $o(t)$  many indices $1\leq i\leq t$, $H_i$ contains $m$ edge-disjoint perfect matchings.

 All in all, we obtain that whp $H$ contains at least $(t-o(t))m=(1-o(1))\binom{n-1}{k-1}p$
  edge-disjoint perfect matchings as required. This completes the proof.
\end{proof}

\section{Tools}
\label{sec:tools}

In what follows, we present some tools that will be useful in our proofs.
\subsection{Threshold for containing a perfect matching}

A key ingredient in our proof is the following $k$-partite version
of the main result in \cite{JKV} which is obtained by a
straightforward modification of its proof (a full proof can be found in \cite{GM}).

\begin{theorem}
  \label{JKV kpartite}
 Let $k$ be a positive integer and let
 $p=\omega\left(\frac{\log n}{n^{k-1}}\right)$. Then, with probability at least $1-n^{-\omega(1)}$, a hypergraph
 $H^k_{n\times k,p}$ contains a perfect matching.
\end{theorem}

\subsection{Sum of independent random variables}

We make use of the following concentration result from \cite{McD}
(Theorem 2.5).

\begin{theorem}
  \label{colin}
  Let $X_1,\ldots,X_t$ be independent random variables, with
  $a_k\leq X_k\leq b_k$ for each $k$, for suitable $a_k$ and $b_k$.
  Let $S_t:=\sum X_k$ and let $\mu:=\mathbb{E} [S_t]$. Then, for
  each $\lambda \geq 0$,

  $$\Pr\left[|S_t-\mu|\geq \lambda\right]\leq 2e^{-2\lambda^2/\sum
  (b_k-a_k)^2}.$$
\end{theorem}

\subsection{Talagrand's inequality}

We also use the following version of Talagrand's inequality  \cite{ChromaticBook} (we remark that in fact, stronger versions exist, see e.g. \cite{MR}, with weaker assumptions on the constants in the bounds bellow, but the following version suffices for our needs).

\begin{theorem}\label{Talagrand}
Let $X$ be a non-negative random variable, not identically $0$,
which is determined by $n$ independent trials $T_1,\ldots,T_n$, and
satisfying the following for some $c,r>0$:
\begin{enumerate} [$(i)$]
\item changing the outcome of any one trial can affect $X$ by at
most $c$, and

\item for any $s$, if $X\geq s$ then there is a set of at most $rs$
trials whose outcomes certify that $X\geq s$.
\end{enumerate}
Then for any $0\leq t\leq \mathbb{E}(X)$,
$$\Pr\left[|X-\mathbb{E}(X)|>t+60c\sqrt{r\mathbb{E}(X)}\right]\leq
4\exp\left(-\frac{t^2}{8c^2r\mathbb{E}(X)}\right).$$
\end{theorem}

%\subsection{Concentration of random polynomials}
%
%We also make use of the so called ``ploynomial method" introduced by
%Kim and Vu \cite{KV}. Let $H$ be a hypergraph, and let $t_i,i\in
%V(H)$ be mutually independent indicator random variables with
%$\mathbb E[t_i]=p_i$. We assume that each edge has at most $k$
%elements (it might be empty as well). Consider a polynomial
%
%$$Y_H=\sum_{e\in E(H)}\prod_{i\in e}t_i,$$
%
%where for an empty edge $e$ we define $\prod_{i\in e}t_i=1$.
%
%For each subset $A$ of $V(H)$, $H_A$ is defined as follows:
%$V(H_A)=V(H)\setminus A$ and $E(H_A)=\{B\subseteq V(H_A):B\cup A\in
%E(H)$. Formally, we can write
%$$Y_{H_A}=\sum_{e, A\subseteq e}\prod_{i\in e\setminus Aa}t_i,$$
%and note that $Y_{H_A}$ is just the partial derivative of $Y_H$ with
%respect to $\{t_i:i\in A\}$.
%
%Now, let $\mathbb{E}_i(H)=\max \{\mathbb{E}(Y_{H_A}):A\subseteq
%V(H),|A|=i\}$ and set $\mathbb{E}(H)=\max\{\mathbb{E}_i(H):i\geq
%0\}$ and $\mathbb{E}'(H)=\max\{\mathbb{E}_i(H):i\geq 1\}$. Now we
%can state the following result due to Kim and Vu:
%
%\begin{theorem}
%  \label{KimVu} In this setting, for any positive number $\lambda>1$ and $a_k=8^k(k!)^{1/2}$ we have
%$$\Pr\left[|Y_H-\mathbb{E}(Y_H)|>a_k\left(\mathbb{E}(H)\mathbb{E'}(H)\right)^{1/2}\lambda^k\right]=O\left(\exp\left(-\lambda+(k-1)\log
%n\right)\right),$$
%\end{theorem}

\section{Proof of Theorem \ref{main1}}
\label{sec:main}
In this section we prove the following, seemingly weaker statement. Then, we show how to derive Theorem \ref{main1} by a simple application of Markov's inequality.

\begin{theorem}
\label{main:weak}
 Let $k\geq 3$ be a positive integer and let
$\frac{\log^{4k}n}{n^{k-1}}\leq p\leq \log^{20k}/n^{k-1}$. Then, whp a
hypergraph $H^k_{n\times k,p}$ contains
$(1-o(1))n^{k-1}p$ edge-disjoint perfect matchings.
\end{theorem}

\begin{proof}[Proof of Theorem \ref{main:weak}] Let $p$ be as in the assumption of the theorem. Let
$\varepsilon>0$, and let $\delta>0$ be a sufficiently small constant. Let
$\beta=10k \delta^2$, $\alpha=\frac{1}{\log^3n}$, and
$\ell$ be an integer such that $(1-\delta+\beta)^{\ell}=\alpha$ (we
omit flooring and ceiling signs as all of our proofs are
asymptotic and this will not harm our calculations).

We describe a randomized algorithm for generating a subhypergrpah
$H'$ of $H^k_{n\times k,p}$ which consists of
$N:=(1-\varepsilon)n^{k-1}p$ edge-disjoint perfect matchings $M_0,\ldots,M_{N-1}$.
Moreover, we show that the algorithm succeeds with a sufficiently
high probability, as required in the statement. As described in the outline (Section \ref{sec:outline}), our algorithm is divided into the following two main phases.

{\bf Phase 1.} Building $N$ edge-disjoint matchings, each of which of size $(1-\alpha)n$.

{\bf Phase 2.} Completing each of the matchings into a perfect matching, keeping all of them edge-disjoint.

\subsection{Phase 1.}

Phase 1. is in fact the heart of the proof and contains all the ideas which are needed for us. We divide Phase 1 into $N$ \emph{rounds}, where in each round $i$ we find a
matching $M_i$ which does not use edges from $\bigcup_{j<i}
M_j$. For $0\leq i\leq N-1$, round $i$ is divided into $\ell$
\emph{steps}, and for every $i$ and $j$ we refer to the $j$-th step of the $i$-th round
as \emph{time step} $ij$. In each time step $ij$ we form a
matching $M_{ij}$ by adding a small ``bite" to a previous matching $M_{i(j-1)}$, until
a matching $M_i:=M_{i\ell}$ of size $(1-\alpha)n$ is obtained. Initially, we set
$j=-1$ and for every $0\leq i<N$ we set $M_{ij}:=\emptyset$.
In order to build the $M_{ij}$s we expose ``relevant" edges with a carefully chosen probability
$q_{ij}$ (to be determined throughout the algorithm).

Before giving a formal description of the algorithm we introduce
some useful notation. An edge $e\in V_1\times\ldots \times V_k$ is called
\emph{relevant} at time step $ij$ if $e\cap
V(M_{i(j-1)})=\emptyset$. That is, if none of its vertices is
incident with an edge of $M_{i(j-1)}$ (note that we completely ignore the fact that few of those edges may belong to other matchings). In each time step $ij$, for
every $1\leq m\leq k$, let $U^{ij}_m:=V_m\setminus V(M_{i(j-1)})$ be the subset
of the \emph{uncovered} vertices of $V_m$ and observe that all these
sets are of the exact same size $n_{j}:=n-|M_{i(j-1)}|$. Let $R_{ij}$
denote the set of all relevant edges at time step $ij$, and note
that $R_{ij}$ corresponds to a complete $k$-partite hypergraph with
$U^{ij}_1\ldots U^{ij}_k$ as its parts.

{\bf The algorithm}  For
$i=0,1\ldots,N-1$ and $j=0,\ldots, \ell-1$ do the following. Randomly assign edges of $R_{ij}$
with color $ij$, independently, with probability $q_{ij}:=\delta
n_j^{-(k-1)}$ (note that an edge can be assigned with more than one
color). Among all the edges colored $ij$, choose a matching $M$ of
size exactly $(\delta-\beta)n_{j}$, uniformly at random. Then,
update $M_{ij}:=M_{i(j-1)}\cup M$. If such a matching does not
exist, then the algorithm reports an error and terminates.

\bigskip

Before we analyze to algorithm, let us make the following easy observation:
\begin{observation}\label{obs:1}
  Throughout the algorithm, assuming that it does not terminate, for every $0\leq j\leq \ell-1$ we have $n_j=(1-\delta+\beta)^jn$.
\end{observation}

\begin{proof}
Since for every $j\leq \ell-1$, in time step $ij$ we enlarge $M_{i(j-1)}$ by
$(\delta-\beta)n_j$, it follows that $n_{j+1}=n_j-(\delta-\beta)n_j=(1-\delta+\beta)n_j$. The observation now follows by a simple induction.
\end{proof}

First, we show that if the algorithm does not terminate, then whp all the obtained matchings are edge-disjoint.

\begin{claim}
  \label{matchings disjoint}
  All the $M_i$-s are edge disjoint whp.
\end{claim}

\begin{proof}
Recall that for each $i$, $M_i$ is formed by edges which are
colored $ij$ for some $j\leq \ell$ and that in each time step $ij$
an edge is colored $ij$ with probability $q_{ij}\leq p$. Moreover,
note that if at the end of the process each edge is assigned with at
most one color, then clearly the matchings are disjoint. Therefore, it will be enough to show that the probability for the existence of an edge $e\in V_1\times \ldots \times V_k$ which is being assigned at least two colors is $o(1)$.

To this end, observe that since $(1-\delta+\beta)^{\ell}=\alpha$, since $\delta$ and $\beta$ are constants and $\alpha=1/\log^3n$, it follows that $\ell=O(\log\log n)$.
Moreover, there are $T:=N\ell=O(n^{k-1}p\log\log n)=polylog(n)$ many time steps (recall that
we assume an upper bound on $p$), where in each time step $ij$, we color edges with probability $q_{ij}\leq p$ (Observation \ref{obs:1}), and the time steps are independent. Therefore, the probability that there exists an edge
 which is being colored at least twice is at most
$$n^kT^2p^2=\frac{n^kpolylog(n)}{n^{2k-2}}=\frac{polylog(n)}{n^{k-2}}.$$

Since $k\geq 3$, the result follows.
\end{proof}

Note that this is the only place where we use the fact that $k\geq 3$. For $k=2$ there are few overlapps between the matchings and it requires a bit more careful treatment. For an example illustrating how to deal with it, the reader is referred to \cite{FS}.

Second, we show that whp the algorithm described above does not terminate and that the sets $U_m$ as defined in the algorithm enjoys a uniform distribution.

\begin{claim}
  \label{small bite exists}
  For every $0\leq i\leq N-1$ and $j<\ell$, at time step $ij$, with probability $1-n^{-\omega(1)}$ we have a matching of size at
least $(\delta-\beta)n_j$. Moreover, by picking such a matching $M$ uniformly at random, for every $1\leq m\leq k$ we have that $M\cap U^{ij}_m$ is a subset of $U^{ij}_m$ of size $(\delta-\beta)n_j$, chosen according to a uniform distribution.
\end{claim}

\begin{proof}
  In order to prove the first part of Claim \ref{small bite exists} we make use of
  Theorem \ref{Talagrand}. Note that in each time step $ij$, the color
  class $ij$ (that is, the set of all edges which have been colored $ij$ during the algorithm) is distributed as $H^{k}_{n_j\times k,q_{ij}}$,
  with $q_{ij}=\frac{\delta}{n_j^{k-1}}$ and $n_j\geq \alpha n$. Therefore, it is enough to
  show that the probability for $H'=H^k_{m\times
  k,\delta/m^{k-1}}$ not to have a matching of size
  $(\delta-\beta)m$ is $m^{-\omega(1)}$, for every $m\geq \alpha n$.

  Let $X$ be the random variable corresponds to the size of the maximal matching we
  have in $H'$, and let $T_e$, $e\in R_{ij}$ be independent indicator random
  variables for the events ``$e\in E(H')$". Note that $X$ is determined
  by the $T_e$-s and that it trivially satisfies $(i)$ and $(ii)$ of
  Theorem \ref{Talagrand} with respect to $c=r=1$. Therefore, we
  have
  \begin{align}\label{ineq} \Pr\left[|X-\mathbb{E}(X)|>t+60c\sqrt{r\mathbb{E}(X)}\right]\leq
4\exp\left(-\frac{t^2}{8c^2r\mathbb{E}(X)}\right).
\end{align}

Now, for each $e\in R_{ij}$ we say that $e$ is an \emph{isolated}
edge in $H'$ if $e\in E(H')$ and all the vertices in $e$ have degree
exactly $1$ in $H'$. Therefore, we have
\begin{align*}
\Pr\left[e\text{ is isolated}\right]&=
\frac{\delta}{m^{k-1}}(1-\frac{\delta}{m^{k-1}})^{m^k-(m-1)^k}\\
&\geq \frac{\delta}{m^{k-1}}\left(1-\frac{\delta}{m^{k-1}}(m^k-(m-1)^k)\right)\\
&= \frac{\delta}{m^{k-1}}\left(1-\frac{\delta}{m^{k-1}}(km^{k-1}+O(m^{k-2}))\right)\\
&\geq (\delta-\beta/2)m^{-(k-1)},
\end{align*}

where here we made use of the facts $(1-x)^n\geq 1-nx$ for all $x>-1$ and $m^k-(m-1)^{k}=km^{k-1}+O(m^{k-2})$. 

All in all, we obtain that $$\mathbb{E}(X)\geq (\delta-\beta/2)m.$$ The result
now easily follows by plugging this estimate into \eqref{ineq} with (say) $t=\beta \mathbb{E}(X)/10$, using the fact that $m\geq \alpha n=n/\log^3n$.

For the second part of the claim, note that one can relabel the vertices of each $U^{ij}_m$ according to a permutation $\pi_m: U^{ij}_m\rightarrow U^{ij}_m$, chosen uniformly, independently at random. Then, after picking the desired matching $M$, one can assign each vertex $v$ with the ``original" label by applying $\pi^{-1}_m(v)$. Clearly, this procedure gives as a uniformly chosen subset of $U^{ij}_m$, for every $m$, as desired.
\end{proof}

For every $i$, let us denote by $U_i:= V_1\times \ldots \times V_k \setminus V(M_i)$, to be the set of all vertices which are \emph{uncovered} by the matching $M_i$. Observe that $U_i:=S_1\times \ldots \times S_k$ for some $S_j\subseteq V_j$, each of which is of size exactly $\alpha n$.
The following claim which follows almost immediately from Claim \ref{small bite exists} and will serve us in Phase 2.

\begin{claim}
  \label{garbage is uniform}
At the end of the algorithm, whp we have that for every $e\in V_1\times \ldots \times V_k$, the number of indices $i$ for which $e\in U_i$ is at most $2\alpha^kN$.
\end{claim}

\begin{proof}
  Let $e\in U_i$, where $U_i:=S_1\times \ldots \times S_k$ as described above. By Claim \ref{small bite exists}, we conclude that each of the $S_m$, $1\leq m\leq k$, is a subset of $V_m$ of size precisely $\alpha n$, chosen uniformly, independently at random. Therefore, the probability for $e\in U_i$ is $\alpha^k$, and the expected number of $i$s for which $e\in U_i$ is $\alpha^k N$. Since the rounds run independently, by Chernoff's bounds we obtain that the probability of $e$ to be in more than $2\alpha^kN$ such $U_i$s is at most
  $$e^{-\Theta(\alpha^kN)}=e^{-\Theta(\alpha^kn^{k-1}p)}=o(n^{-k}).$$
  Note that here we make use of the fact that $p=\log^{C}n/n^{k-1}$, where $C$ depends on $k$.
  Therefore, by taking the union bound over all possible $e\in V_1\times \ldots \times V_k$, we obtain the desired.
\end{proof}

To conclude, let $H'$ be the hypergraph consisting of all the edges which
have received any color during the algorithm. We show that indeed
$H'$ can be coupled as a subhypergraph of $H=H^k_{n\times k,(1-\varepsilon/2)p}$. To
this end, it will be convenient to introduce some notation. For
every $e\in V_1\times\ldots\times V_k$, let us define $R(e):=\{ij: e
\in R_{ij}\}$ (note that $R(e)$ is a random variable, and that for every $i,j$, at the beginning of time step $ij$ it is already known whether $ij\in R(e)$ or not). Observe that for each $ij\in R(e)$, at time step $ij$
we try to assign $e$ with the color $ij$, with probability $q_{ij}$,
independently at random. Let $\gamma>0$ be a sufficiently small constant (to be determined later), since the probability of $e$ not being
colored with any color, conditioned on $R(e)$, is $1-q_e:=\prod_{ij\in R(e)}(1-q_{ij})$, it
follows that $\Pr\left[e\in E(H')\right]=q_e\leq (1+\gamma)\sum_{ij \in
R(e)}q_{ij}$ (here we use the fact that $(N\ell)^2p=o(1)$, as $p\leq \log^{20k}n/n^{k-1}$).
Therefore, in order to show that one can generate $H'\subseteq H$, all we need to show is
that by following our algorithm, whp
we have $q_e\leq (1-\varepsilon/2)p$ for every  $e$. This is done in the
following (quite) technical claim.

\begin{claim}
  \label{claim1}
 With probability $1-n^{-\omega(1)}$ we have that $q_e\leq (1-\varepsilon/2)p$ for
 every $e\in V_1\times \ldots \times V_k$.
\end{claim}

\begin{proof}
  For every $0\leq i\leq N-1$ and $e\in V_1\times\ldots\times V_k$ consider the random variable
  $$\omega_i(e)=\sum_{j:ij\in R(e)}q_{ij},$$
  and observe that
  $$q_e\leq (1+\gamma)\sum_i\omega_i(e).$$

  Moreover, by the description of the algorithm it follows that $$q_{ij}\leq \delta/(\alpha n)^{k-1}=\delta\log^{3k-3}n/n^{k-1}$$
  for every $i$ and $j$, and therefore (deterministically) we have

  $$\omega_i(e)\leq \ell \delta\log^{3k-3}n/n^{k-1}\leq \log^{3k-2}n/n^{k-1}.$$

  In order to complete the proof we need to show two things. First,
  we show that $\mathbb{E}(\omega_i(e))\leq (1-\gamma)n^{-(k-1)}$ (and therefore, we obtain $\mathbb{E}(q_e)\leq (1-\gamma^2)n^{-(k-1)}N\leq (1-\varepsilon)p$). Then, using standard concentration bounds,
  we show that with probability $1-n^{-\omega(1)}$
  we have $q_e\leq (1-\varepsilon/2)p$.

  {\bf Estimating $\mathbb{E}(\omega_i(e))$.} Note
  that $ij\in R(e)$ if and only if $e$ is relevant at time step
  $ij$, and that by Claim \ref{small bite exists} we observe that at each time step $ij$ of the
  algorithm, any vertex $v\in U_m$ is being ``matched" with
  probability $\delta-\beta$, where vertices from different
  $U_m$-s are independent. Therefore, at each time step, the probability for a relevant $e$ to stay relevant is
  $(1-\delta+\beta)^k$, and the probability for not staying
  relevant is $1-(1-\delta+\beta)^k$, which is roughly
  $k(\delta-\beta)$ (recall that $\delta$ is sufficiently small).

  Now, for each $j\leq \ell-1$, let us denote by $A_j$ the event
  ``$j$ is the maximal index for which $e$ is relevant at time step
  $ij$", and observe that
\begin{align}\label{eq:expectation}\mathbb{E}(\omega_i(e))&=\sum_{j=0}^{\ell-1}\Pr\left[A_j\right]\sum_{s=0}^jq_{is}\nonumber \\
&=\sum_{s=0}^{\ell-1}q_{is}\sum_{j=s}^{\ell-1}\Pr\left[A_j\right].
\end{align}

Note that $\sum_{j=s}^{\ell-1}\Pr\left[A_j\right]$ is the probability for an edge to be relevant \emph{at least} $s$ steps, and therefore is equal to $(1-\delta+\beta)^{ks}$.
   Combining it with \eqref{eq:expectation}, we get that
   $$\mathbb{E}(\omega_i(e))=\sum_{s=0}^{\ell-1}q_{is}(1-\delta+\beta)^{ks}.$$
   Recalling that $q_{is}=\delta/n_s^{k-1}$, and that $n_s=(1-\delta+\beta)^{s}n$ (Observation \ref{obs:1}), we obtain that
   \begin{align}\mathbb{E}(\omega_i(e))&=\sum_{s=0}^{\ell-1}\frac{\delta}{(1-\delta+\beta)^{s(k-1)}n^{k-1}}(1-\delta+\beta)^{sk}\nonumber \\
   &= \frac{\delta}{n^{k-1}}\sum_{s=0}^{\ell-1}(1-\delta+\beta)^s \nonumber \\
   &=\frac{\delta}{n^{k-1}}\frac{1-(1-\delta+\beta)^{\ell}}{\delta-\beta}.\nonumber
   \end{align}
   (The second equality is just the sum of a geometric series.)

Now, since $\delta$ is a sufficiently small constant, since $\beta/\delta$ tends to zero with $\delta$, and by the way we chose $\ell$, we obtain that the right hand side in the above equality is at most $(1-\gamma) n^{-(k-1)}$. Therefore, we obtain that
$$\mathbb{E}(q_e)\leq (1+\gamma)\sum_i \mathbb{E}(\omega_i(e))\leq (1-\gamma^2)n^{-(k-1)}N\leq (1-\varepsilon)p.$$

{\bf Showing that $q_e\leq p$ with a sufficient probability.} Consider the random
variables $\omega_i(e)$, $0\leq i\leq N-1$ and observe that they are mutually
independent. Moreover, as noted above, $0\leq \omega_i(e)\leq
\log^{3k-2}n/n^{k-1}$ for every $i$. Now, let
$\omega^*=\sum_{i=0}^{N-1} \omega_i(e)$ and recall that $\mathbb{E}(\omega^*)\leq
(1-\varepsilon)p$. By applying Theorem \ref{colin} to $\omega^*$
we obtain

\begin{align*}
\Pr\left[|\omega^*-\mathbb{E}(\omega^*)|\geq \varepsilon
p/2\right]&\leq 2\exp\left(-\frac{0.5\varepsilon^2
p^2}{N\log^{6k-4}nn^{-2(k-1)}}\right)\\
&\leq 2\exp\left(-\frac{ 0.5\varepsilon^2 pn^{k-1}}{(1-\varepsilon)\log^{6k-4}}\right)=n^{-\omega(1)}.
\end{align*}
Taking the union bound over all possible $e\in V_1\times\ldots\times
V_k$ we obtain the desired. This completes the proof of the claim.
\end{proof}

\subsection{Phase 2.}

In this phase, we want to show that one can complete each of the $M_i$s from Phase 1. into a perfect matching in an edge-disjoint way. To this end, let $U_i$, $0\leq i\leq N-1$ denote the set of all vertices which are uncovered by $M_{i}$, and let $q=\log^5 n/n^{k-1}$. Observe that $|U_i|=k\alpha n=kn/\log^3n$, and that $q=\omega(\log |U_i|/|U_i|^{k-1})$.
For every $i$, let us expose all the $k$-tuples $e\in U_i$, with probability $q$, independently at random, and denote the resulting graph as $H_i$. Clearly, $H_i=H^k_{|U_i|,q}$. Now, by applying Theorem \ref{JKV kpartite} to $H_i$ and by taking the union bound over all $i$, it follows that $H_i$ contains a perfect matching $Q_i$ for all $i$.
Let $M_i:=M_i\cup Q_i$, and observe that each of the $M_i$s is a perfect matching of $H$.

In order to complete the proof, we need to show:
\begin{enumerate}
  \item all the $M_i$s are edge-disjoint, and
  \item no edge accumulated a weight of more than $\varepsilon p/3$.
\end{enumerate}

$1.$ follows in a similar way as in Claim \ref{matchings disjoint}. For $2.$, note that by Claim \ref{garbage is uniform} we have that no edge belongs to more than $2\alpha^kN$ many $U_i$s. Therefore, since we expose edges of $U_i$s with probability $q$, every edge accumulates a weight of at most $2\alpha^kNq=2n^{k-1}p\log^5 n/\log^{3k}n^{k-1}=o(p)$, as desired. This completes the proof of Theorem \ref{main:weak}.
\end{proof}

\subsection{Derivation of Theorem \ref{main1}}

In this section we show how to derive Theorem \ref{main1} from Theorem \ref{main:weak}.

Let $p\geq \log^{20k}n/n^{k-1}$ and let $r\in \mathbb{N}$ be an integer for which
$\frac{\log^{4k}n}{n^{k-1}}\leq p/r\leq
\frac{\log^{20k}n}{n^{k-1}}$. Now, expose the edges of
$H^k_{n\times k,p}$, and for each exposed edge, immediately assign with a color from
$[r]$, independently, uniformly at random.

Observe that for each color
class $i\in [r]$, the corresponding hypergraph $H_i$ is distributed as
$H^k_{n\times k,p/r}$ and that for every $i\neq j$, $E(H_i)\cap E(H_j)=\emptyset$. Therefore, by applying Theorem \ref{main:weak} to $H_i$, with probability $1-o(1)$ $H_i$ contains $(1-o(1))n^{k-1}p/r$ edge-disjoint perfect matchings. Using Markov's inequality we obtain
that whp for $r-o(r)$ hypergraphs $H_i$, the above holds, and therefore, $H=\cup H_i$ contains $(1-o(1)rn^{k-1}p/r=(1-o(1))n^{k-1}p$ edge-disjoint perfect matchings as desired.

{\bf Acknowledgment.} We would like to thank the referees for many valuable comments.

\end{document}